\documentclass[12pt]{article}
\usepackage{amsmath}
\usepackage{graphicx}
\usepackage{pslatex}
\usepackage{amssymb,amsfonts,latexsym,amstext,amsmath,amsthm,enumerate}
\newtheorem{theorem}{Theorem}
\newtheorem{lemma}[theorem]{Lemma}
\newtheorem{corollary}[theorem]{Corollary}

\newtheorem{remark}[theorem]{Remark}
\newtheorem{definition}[theorem]{Definition}

\newcommand\Zp{\makebox[.75em]{$1$}}
\newcommand\Zm{\makebox[.75em]{$-$}}
\newcommand\Zz{\makebox[.75em]{$0$}}

\newcommand{\BR}{\{e^{\frac{2\pi i}{3}},e^{-\frac{2\pi i}{3}}\}}

\newenvironment{myind}[1]%
{\begin{list}{}%
         {\setlength{\leftmargin}{#1}}%
         \item[]%
}
{\end{list}}
\begin{document}
\title{On unit weighing matrices with small weight}
\author{Darcy Best, Hadi Kharaghani\footnotemark, Hugh Ramp\\
Department of Mathematics \& Computer Science\\
University of Lethbridge\\
Lethbridge, Alberta, T1K 3M4\\
Canada\\
 darcy.best@uleth.ca, kharaghani@uleth.ca, hugh.ramp@uleth.ca }

\maketitle
\begin{abstract}
The structure of unit weighing matrices of order $n$ and  weights  2, 3 and 4 are studied.  
We show that the number of inequivalent unit weighing matrices $UW(n,4)$ depends on the
 number of decomposition of $n$ into sums of non-negative multiples of some specific positive
 integers. Two interesting sporadic cases are presented in order to demonstrate the complexities involved 
in the classification of weights larger than 4. \end{abstract}

\footnotetext[1]{Supported by an NSERC-Group Discovery Grant. Corresponding author.}

\section{Introduction}
A weighing matrix of order $n$ and weight $p$, denoted $W(n,p)$, is a  $(0,\pm 1)-$matrix of order 
$n$ such that $WW^T =pI_n$.  For $n = p$, the matrix is a {\it Hadamard matrix}. 
%Let  $\mathbb{T}$ be the set of  all complex numbers with absolute value one. 
The natural extension of Hadamard matrices 
to the unit Hadamard matrices, i.e., the Hadamard matrices whose entries are complex numbers on the unit circle, has inspired the idea of unit weighing matrices, 
which is defined as follows.

\begin{definition}
 A  square matrix $W=[w_{ij}]$, $w_{ij}\in \mathbb{T} \cup \{0\}$, of order $n$ and weight $p$, 
denoted $UW(n,p)$, for which $WW^*=pI_n$, where $^*$ is the Hermitian transpose and $\mathbb{T}$ is the set of 
all complex numbers of absolute value one, is called a {\bf unit weighing matrix}.\end{definition}

Weighing matrices have been studied quite extensively and there are over 200 reviewed papers
on these matrices in the literature. We refer the interested reader to \cite{ck-crc} for general study of weighing matrices.
The structure of weighing matrices with large weights (but not Hadamard) has been studied in \cite{craigen}. 
All real Hadamard matrices (i.e. weighing matrices of full weights) have been classified up to order 32 (see \cite{hk32,H32}).
A lot of interest has been shown in circulant weighing matrices and a fair number of orders and weights have been classified for 
this subset of weighing matrices, see for example \cite{ss} and the references therein. Circulant weighing matrices of 
small weights are studied by Strassler \cite {s1,s2,s3} 
and in Epstein's Master's thesis \cite {e1}.

Unit weighing matrices appear quite naturally in the study of mutually unbiased bases in diverse areas of quantum-informatic applications 
(see \cite{durt}) and in signal processing (see \cite{adam}).
In general, unit weighing matrices have very complex structures and are not easy to classify. 
For example, in the special case of unit Hadamard matrices, the full structure up to order 5 \cite{H5} 
is known and a classification of order 6 has shown to be stubbornly complicated \cite{t}. 
Naturally, in general the study of the structure of unit weighing matrices is much harder.  

In this paper, the complete structure 
of unit weighing matrices of weight less than 5 is provided. Two examples to show the complexities involved in weight 5 are also included.

Throughout the paper, we use ``$-$'' to denote ``$-1$'' and use {\it unimodular number} for the elements of $\mathbb{T}$.

\section{Equivalence}\label{equiv}

\begin{theorem} \label{th:equiv}
For a given unit weighing matrix, applying any of the following operations will result in a unit weighing matrix:
\begin{enumerate}
 \item[(T1)] Permuting the rows
 \item[(T2)] Permuting the columns
 \item[(T3)] Multiplying a row of the matrix by a number in  $\mathbb{T}$
 \item[(T4)] Multiplying a column of the matrix by a  number in $\mathbb{T}$
 \item[(T5)] Taking the Hermitian transpose
 \item[(T6)] Conjugating every entry in the matrix
\end{enumerate}

\end{theorem}

Note that by applying $(T5)$ followed by $(T6)$, we have that the transpose of a unit weighing matrix is also a unit weighing matrix.

\begin{definition}
 Two unit weighing matrices $W_1$ and $W_2$ are {\bf equivalent} if one can be obtained from the other 
by performing a finite number of  operations $(T1)$, $(T2)$, $(T3)$ and $(T4)$ to it.
\end{definition}

Note that $(T5)$ and $(T6)$ are excluded from the definition in order to maintain consistency between the 
definitions of equivalence for Hadamard matrices and weighing matrices.

Inequivalent (real) weighing matrices have been studied quite extensively for many special cases. 
All weighing matrices having weights less than 6 were classified in \cite{seberry} (see Remark \ref{rem:weight-5-fix} below). 

A lot of interest has been shown in circulant weighing matrices, see \cite{ss} and the reference therein. A fair number of orders and weights have been classified for 
this subset of weighing matrices. Circulant weighing matrices of small weights are extensively studied by Strassler \cite {s1,s2,s3} 
and Epstein's Master's thesis \cite {e1}.

In order to study the number of inequivalent unit weighing matrices, we define the following ordering, $\prec$, 
on the elements of $\mathbb{T}\cup \{0\}$.   \begin{enumerate}
   \item $e^{i\theta} \prec 0$ for all $\theta$
   \item $e^{i\theta} \prec e^{i\phi} \iff 0 \leq \theta < \phi < 2\pi$
  \end{enumerate}

\begin{definition}
 We say that a unit weighing matrix, $UW(n,w)$ is in {\bf standard form} if the following conditions apply:

 \begin{enumerate}
  \item[(S1)] The first non-zero entry in each row is 1.
  \item[(S2)] The first non-zero entry in each column is 1.
  \item[(S3)] The first row is $w$ ones followed by $n-w$ zeros.
  \item[(S4)] The rows are in lexicographical order according to $\prec$.
 \end{enumerate}
\end{definition}

To clarify the ordering in $(S4)$ (say we are interested in row $i$ and row $j$), we denote row $i$ by $R_i=\left(a_1,a_2,\cdots,a_n\right)$ and row $j$ by $R_j=\left(b_1,b_2,\cdots,b_n\right)$ and let $k$ be the smallest index such that $a_k \neq b_k$. Then $R_i < R_j \iff a_k \prec b_k$.

\begin{theorem}\label{th:standard}
 Every unit weighing matrix is equivalent to a unit weighing matrix that is in the standard form.
\begin{proof}
 Let $W$ be a unit weighing matrix of weight $w$. Let $r_i \in \mathbb{T}$ be the first non-zero entry in row $i$. Multiply each row $i$ by $\overline{r_i} \in \mathbb{T}$, so that the condition $(S1)$ holds. For column $j$, let $c_j \in \mathbb{T}$ be the first non-zero entry  in the transformed matrix. Multiply each column $j$ by $\overline{c_j} \in \mathbb{T}$, which satisfies condition $(S2)$. Permute the columns so that the first row has $w$ non-zeros (each of which must be one since $(S2)$ is satisfied) followed by $n-w$ zeros, which satisfies $(S3)$. Finally, sort the rows of the matrix lexicographically with the ordering $\prec$. Note that the first row will not have moved since it is the least lexicographic row in the matrix. The transformed matrix now satisfies condition $(S4)$, and hence, is in standard form.
\end{proof}
\end{theorem}

It is important to note that two matrices that have different standard forms may be equivalent to one another. Studying the number of standardized weighing matrices will lead to an upper bound on the number of inequivalent unit weighing matrices.

\section{The existence of unit weighing matrices}

We need the following definition in order to determine the existence of certain unit weighing matrices. 
\begin{definition}
 Let $S \subset \mathbb{T}$. $S$ is said to have {\bf m-orthogonality} if there are $a_1,...,a_m,b_1,...,b_m \in S$ such that $\sum_{i=1}^m c_i = 0$, where $c_i=a_i\overline{b_i}$.
\end{definition}
 
We will be using the following results for a few small values of $m$ in this paper. Each may be verified easily, so we do not include their proofs here.

$
\begin{array}{crl}
  ~~~~ & m=0 & \text{Trivially orthogonal} \\
  & m=1 & \text{No } S \text{ has } 1\text{-orthogonality} \\
  & m=2 & \text{If } c_1+c_2=0\text{, then  }  c_1 = -c_2 \\
  & m=3 & \text{If } c_1+c_2+c_3=0\text{, then  }\\& & c_1=e^{iq},  c_2=e^{i\left(\frac{2\pi}{3}+q\right)},  c_3=e^{i\left(\frac{4\pi}{3}+q\right)} \text{ for some real number } q \\
  & m=4 & \text{If } c_1+c_2+c_3+c_4=0\text{, then we may assume that } c_1 = -c_2 \text{ and } c_3 = -c_4
\end{array}
$

%Note that a set $S$ may have $m$-orthogonality for many values of $m$. For example, the set of all third roots of unity have $m$-orthogonality for all multiples of 3, whereas the set of all sixth roots of unity have $m$-orthogonality for all $m \neq 1$.

We begin by extending a result of \cite[Proposition 2.5]{gse} to unit weighing matrices.
\begin{lemma}
\label{lem:n-2z-orth}
 If there is a $UW(n,w)$ and $n > z^2-z+1$, where $z = n-w$, then there is a set that has $(n-2z)$-orthogonality.
\begin{proof}
 First, note that the cases where $z\leq1$ are straightforward. Now assume $z\geq2$. Through appropriate row and column permuations, we may assume that the first $z$ entries in the first row and first column are $0$.
\begin{itemize}
 \item Let $Z(i,j)$ be the number of zeros in the first $j$ rows of the $i$-{th} column
 \item Let $E(k)$ be the row that contains the last $0$ in column $k$ ({\it i.e.}, $Z(k,j) = w$ for all $j \geq E(k)$ and $Z(k,j) < w$ for all $j < E(k)$).
\end{itemize}
By construction, $E(1)=z$. We know that $1 \leq Z(2,E(1)) \leq z$, so by appropriate row permutations, the next $z-Z(2,E(1))$ rows will have a zero in the second column. This implies $$E(2) = E(1) + (z-Z(2,E(1))) = 2z-Z(2,E(1)) \leq 2z-1.$$ Furthermore, $1 \leq Z(3,E(2)) \leq z$. We once again perform row permutations so that the next $z-Z(3,E(2))$ rows have a zero in the third column, so $$E(3) = E(2)+(z-Z(3,E(2))) \leq (2z-1)+(z-Z(3,E(2))) \leq 3z - 2.$$ In general, following this process, we have $$E(k) \leq kz-(k-1)$$ for $k \leq z$. So this gives us $E(j) \leq z^2-(z-1)$ for $j \leq z$. Thus, if we examine row $z^2-z+2$, we know that the first $z$ columns already have $z$ zeros in them, thus, all $z$ zeros must appear in the last $n-z$ columns of that row. The set of entries in the first row and that row has $(n-2z)$-orthogonality. It is noteworthy to mention that $n > z^2-z+1$ implies that $n-2z \geq 0$. 
\end{proof}
\end{lemma}
\begin{corollary} [Geramita-Geramita-Wallis]
 For odd $n$, a necessary condition that a $W(n,w)$ exists is that $(n-w)^2-(n-w)+1\ge n$.
\end{corollary}
\begin{proof}
 Let $z=n-w$. For odd $n$, $n-2z=2w-n$ is odd, but $\{\pm 1\}$ does not have $(n-2z)$-orthogonality. The result follows from Lemma \ref{lem:n-2z-orth}.
\end{proof}

% To further characterize generalized weighing matrices, we define the following multi-set
%$$\Lambda(W) = \left\{w_{ij}\overline{w_{kj}}w_{kl}\overline{w_{il}} : W = [w_{ij}], 1 \leq i,j,k,l \leq n\right\}$$

%We note that the six operations in Theorem \ref{th:equiv} do not affect this multiset, so $\Lambda$ is invariant with respect to these operations. Thus, two weighing matrices that have different $\Lambda$ cannot be equivalent. This has been used to study complex Hadamard matrices and is called the Haagerup invariant.

\subsection{Existence of $UW(n,1)$}

Any weighing matrix of weight 1 is equivalent to the identity matrix. Thus, $UW(n,1)$ exists for every $n \in \mathbb{N}$.

\subsection{Existence of $UW(n,2)$}

We begin with a non-existence of a unit weighing matrix.

\begin{lemma}\label{lem:cw-3-2}
 There is no $UW(3,2)$.
\begin{proof}
By Lemma \ref{lem:n-2z-orth}, the existence of a $UW(3,2)$ would imply the existence of a set having 1-orthogonality.
\end{proof}
\end{lemma}

This leads us to the following theorem.

\begin{theorem}
 A $UW(n,2)$ exists if and only if $n$ is even. Moreover, there is exactly one inequivalent class of $UW(n,2)$ for each even $n$.
\begin{proof}
 Let $W$ be a $UW(n,2)$. By Theorem \ref{th:standard}, we may transform $W$ into a weighing matrix in standard form (we will call this matrix $W'$). Thus, the first two entries of the first column and first row are ones. The second entry in the second row must be $-1$. So we have that our matrix is of this form:

\renewcommand{\arraystretch}{2}
$$\newcommand*{\temp}{\multicolumn{1}{r|}{}}
W'=\left(\begin{array}{clc}
\begin{array}{lr}
1 & 1 \\
1 & - 
\end{array}

&\temp &0 \\ \cline{1-3}
0 &\temp &W'' \\ 

\end{array}\right)
$$

\renewcommand{\arraystretch}{1}
\noindent where $W''$ is  a $UW(n-2,2)$. We may now use the same process on $W''$ and continue until we arrive at the bottom right corner. If $n$ is even, then we can complete the matrix. However, if $n$ is odd, then the process ends with a $3 \times 3$ block which must be a $UW(3,2)$, but we know from Lemma \ref{lem:cw-3-2} that this does not exist. Thus, there is no $UW(n,2)$ for $n$ odd. Since the number of inequivalent weighing matrices is bounded above by the number of standardized matrices and there is only one standardized matrix, every weighing matrix of order $n$ and weight 2, for $n$ even, is equivalent to

$$\left(
\begin{array}{cc}
 1 & 1 \\
 1 & -\\
\end{array}
\right) \oplus \cdots \oplus \left(
\begin{array}{cc}
 1 & 1 \\
 1 & - \\
\end{array}
\right)
$$
\end{proof}

The direct sum notation given above is used in the same way as is laid out in \cite{seberry}.

\end{theorem}

\subsection{Existence of $UW(n,3)$}\label{n3}

Weight 3 is the first example of unit weighing matrices where the results are quite different from the  real weighing matrices. For example, in contrast to the fact that a $UW(3,3)$ exists, there is no $W(3,3)$.

\begin{lemma}\label{lem:cw-n-3}
 Any $UW(n,3)$ can be transformed so that the top leftmost submatrix is either a $UW(3,3)$ or a $UW(4,3)$.

\begin{proof}
 By Theorem \ref{th:standard}, we may alter $W$ so that it is in standard form. This means that the second row has three possibilities after further appropriate column permutations (Note that these permutations should leave the shape of the first row intact). When we say that a row is not orthogonal with another row with no further context, it is because it would imply that the set of elements in the two rows would have 1-orthogonality.

   \begin{enumerate}
    \item $
      \left(
       \begin{array}{cccccccc}
         1 & a  & b &0 &0&0 &\dotsb &0
       \end{array}
      \right)
      $
    \item $
      \left(
       \begin{array}{cccccccc}
         1 & a  & 0 &1 &0 &0 &\dotsb &0
       \end{array}
      \right)
      $
    \item $
      \left(
       \begin{array}{cccccccc}
         1 & 0  & 0 &1 &1 &0&\dotsb &0
       \end{array}
      \right)
      $, 1-orthogonality with row 1, so not possible.
   \end{enumerate}

\begin{myind}{1cm}

   For case 1, 3-orthogonality implies $b = \overline{a}$, where $a \in \BR$, and four further cases arise for the third row.

    \begin{enumerate}[(a)]
     \item  $
      \left(
       \begin{array}{cccccccc}
         1 & c  & d &0 &0&0 &\dotsb &0
       \end{array}
      \right)
      $
    \item $
      \left(
       \begin{array}{cccccccc}
         1 & c  & 0 &1 &0 &0 &\dotsb &0
       \end{array}
      \right)
      $
    \item $
      \left(
       \begin{array}{cccccccc}
         1 & 0  & c &1 &0 &0&\dotsb &0
       \end{array}
      \right)
      $
    \item $
      \left(
       \begin{array}{cccccccc}
         1 & 0  & 0 &1 &1 &0&\dotsb &0
       \end{array}
      \right)
      $, 1-orthogonality with row 1, so not possible.
    \end{enumerate}
  
    For case (b), we have $c = -1$ by orthogonality with the first row and $c = -a$ by orthogonality with the second row. Similarly, in case (c), we have $c = -1$ and $c = -\overline{a}$. Both of these are not possible. However, case (a) produces a viable option when $c=\overline{d}=\overline{a}$, finishing case 1 and implying that the top $3\times3$ submatrix is a $UW(3,3)$ of the following form:

$$
\left(
\begin{array}{ccc}
 1 & 1       &1 \\
 1 & a       &\overline{a} \\
 1 & \overline{a} &a       
\end{array}
\right)
$$

Note that if $a=e^{-\frac{2\pi i}{3}}$, then swap rows 2 and 3, so we may assume $a=e^{\frac{2\pi i}{3}}$.

\end{myind}

    For case 2, $a=-1$ and we have six subcases for the third row:

\begin{myind}{1cm}
     \begin{enumerate}[(a)]
     \item $
      \left(
       \begin{array}{cccccccc}
         1 & b  & c & 0 &0&0 &\dotsb &0
       \end{array}
      \right)
      $, with $-1 \prec b$.
     \item $
      \left(
       \begin{array}{cccccccc}
         1 & b  & 0 &c &0&0 &\dotsb &0
       \end{array}
      \right)
      $
    \item $
      \left(
       \begin{array}{cccccccc}
         1 & b  & 0 &0 &1 &0 &\dotsb &0
       \end{array}
      \right)
      $
    \item $
      \left(
       \begin{array}{cccccccc}
         1 & 0  & b &c &0 &0&\dotsb &0
       \end{array}
      \right)
      $
    \item $
      \left(
       \begin{array}{cccccccc}
         1 & 0  & b &0 &1 &0&\dotsb &0
       \end{array}
      \right)
      $, 1-orthogonality with row 2, so not possible.
     \item $
      \left(
       \begin{array}{cccccccc}
         1 & 0  & 0 &b &1 &0&\dotsb &0
       \end{array}
      \right)
      $, 1-orthogonality with row 1, so not possible.
     \item $
      \left(
       \begin{array}{ccccccccc}
         1 & 0  & 0 &0 &1 &1 & 0&\dotsb &0
       \end{array}
      \right)
      $, 1-orthogonality with row 1, so not possible.

    \end{enumerate}

    In case (a), $b=1$ by orthogonality with row 2 and $b\in\BR$ by orthogonality with row 1. In case (b), $b=-1$ by orthogonality with row 1 and $-b \in \BR$ by orthogonality with row 2. In case (c), $b=-1$ by orthogonality with row 1, which implies row 2 is not orthogonal with row 3. All of which are problems. In case (d), we have a valid configuration by setting $b=c=-1$. We now construct the next row, which gives us four cases:

\end{myind}
\begin{myind}{1.5cm}
      \begin{enumerate}[(i)]
     \item $
      \left(
       \begin{array}{cccccccc}
         0 & 1  & d &f &0&0 &\dotsb &0
       \end{array}
      \right)
      $
    \item $
      \left(
       \begin{array}{cccccccc}
         0 & 1  & d &0 &1 &0 &\dotsb &0
       \end{array}
      \right)
      $, 1-orthogonality with row 3, so not possible.
    \item $
      \left(
       \begin{array}{cccccccc}
         0 & 1  & 0 &d &1 &0&\dotsb &0
       \end{array}
      \right)
      $, 1-orthogonality with row 3, so not possible.
    \item $
      \left(
       \begin{array}{ccccccccc}
         0 & 1  & 0 &0 &1 &1 & 0&\dotsb &0
       \end{array}
      \right)
      $, 1-orthogonality with row 3, so not possible.\end{enumerate}

    In case (i), we have a valid row if $d=-f=-1$, finishing all of the cases above, and giving us a $UW(4,3)$ in the upper left $4\times4$ submatrix of the form:
\end{myind}
$$
\left(
\begin{array}{cccc}
 1 & 1 &1 &0 \\
 1 & - &0 &1 \\
 1 & 0 &- &- \\
 0 & 1 &- &1
\end{array}
\right)
$$
\end{proof}

\end{lemma}

\begin{theorem}
 Every $UW(n,3)$ is equivalent to a matrix of the following form:

$$\left(
\begin{array}{ccc}
 1 & 1       &1 \\
 1 & a       &\overline{a} \\
 1 & \overline{a} &a       
\end{array}
\right) \oplus \cdots \oplus \left(
\begin{array}{ccc}
 1 & 1       &1 \\
 1 & a       &\overline{a} \\
 1 & \overline{a} &a       
\end{array}
\right) \oplus 
\left(
\begin{array}{cccc}
 1 & 1 &1 &0 \\
 1 & - &0 &1 \\
 1 & 0 &- &- \\
 0 & 1 &- &1
\end{array}
\right)
\oplus \cdots \oplus
\left(
\begin{array}{cccc}
 1 & 1 &1 &0 \\
 1 & - &0 &1 \\
 1 & 0 &- &- \\
 0 & 1 &- &1
\end{array}
\right)
$$

where $a=e^{\frac{2\pi i}{3}}$.

\end{theorem}

\begin{proof}
 Let $W$ be a $UW(n,3)$. From Lemma \ref{lem:cw-n-3}, we have that the top leftmost block must be either a $UW(3,3)$ or $UW(4,3)$. From here, we know that the first 3 
(or 4) rows and columns of the matrix are complete, and as such, are trivially orthogonal with the remainder of the matrix, so we may assume that the lower $(n-3)\times(n-3)$ submatrix (or $(n-4)\times(n-4)$ submatrix) is a $UW(n-3,3)$ (or $UW(n-4,3)$). As such, the top submatrix of this matrix will also be of the desired form. Inductively, we continue this process until we reach the end of the matrix. The blocks may then be permuted such that all of the $UW(3,3)$ submatrices appear above the $UW(4,3)$ submatrices.
\end{proof}

\begin{corollary}
 There is a $UW(n,3)$ if and only if $n \neq 5$. The number of equivalence classes  is bounded above by the number of distinct decomposition of $n$ into sums of non-negative multiples of $3$ and $4$.
\end{corollary}

Note that an alternate way to show that $UW(5,3)$ does not exist is to use  Lemma \ref{lem:n-2z-orth}.

\begin{corollary}
 There is a $W(n,3)$ if and only if $n$ is a  multiple of $4$. Moreover, there is only one class of inequivalent matrices.
\end{corollary}

\subsection{Existence of $UW(n,4)$}

Similar to $UW(n,3)$, any $UW(n,4)$ can be defined based on the blocks along the main diagonal. All $UW(n,4)$ are equivalent to a $UW(n,4)$ with diagonal blocks consisting of the following matrices: $W_5$, $W_6$, $W_7$, $W_8$ and $E_{2m}(x)$ where $2 \leq m \leq \frac{n}{2}$ and $x$ is any unimodular number.

$$
W_5=\left(
\begin{array}{rrrrr}
 1 & 1 &1 &1 &0 \\
 1 &  a &\overline{a} &0 &1\\
 1 & \overline{a} &0 &a &\overline{a} \\
 1 & 0 &a &\overline{a} & a \\
 0 & 1 & \overline{a} &a &a 
\end{array}
\right),
W_6=\left(
\begin{array}{rrrrrr}
 1 & 1 &1 &1 &0 &0\\
 1 & a &\overline{a} &0 &1 &0\\
 1 & \overline{a} &a &0 &0 &1\\
 1 & 0 & 0 & - & - & - \\
 0 & 1 & 0 & - & -\overline{a} &-a \\
 0 & 0 & 1 & - & -a &-\overline{a}
\end{array}
\right)
\text{ for } a = e^{\frac{2\pi i}{3}},
$$

 $$W_7=\left(\begin{array}{c}
\Zp\Zp\Zp\Zp\Zz\Zz\Zz\\
\Zp\Zm\Zz\Zz\Zp\Zp\Zz\\
\Zp\Zz\Zm\Zz\Zm\Zz\Zp\\
\Zp\Zz\Zz\Zm\Zz\Zm\Zm\\
\Zz\Zp\Zm\Zz\Zz\Zp\Zm\\
\Zz\Zp\Zz\Zm\Zp\Zz\Zp\\
\Zz\Zz\Zp\Zm\Zm\Zp\Zz
\end{array}\right), 
W_8=\left(\begin{array}{c}
\Zp\Zp\Zp\Zp\Zz\Zz\Zz\Zz\\
\Zp\Zm\Zz\Zz\Zp\Zp\Zz\Zz\\
\Zp\Zz\Zm\Zz\Zm\Zz\Zp\Zz\\
\Zp\Zz\Zz\Zm\Zz\Zm\Zm\Zz\\
\Zz\Zp\Zm\Zz\Zp\Zz\Zz\Zp\\
\Zz\Zp\Zz\Zm\Zz\Zp\Zz\Zm\\
\Zz\Zz\Zp\Zm\Zz\Zz\Zp\Zp\\
\Zz\Zz\Zz\Zz\Zp\Zm\Zp\Zm
\end{array}\right),$$

$$E_{2m}(x) = \left(\begin{array}{rrrrrrrrrrrrrrr} 
1&1&1&1\\1&1&-&-\\1&-&0&0&1&1\\1&-&0&0&-&-\\~&~&1&-&0&0&1&1\\~&~&1&-&0&0&-&-\\~&~&~&~&1&-&0&0\\~&~&~&~&1&-&0&0 \\ ~&~&~&~&~&~&~&~&\ddots \\~&~&~&~&~&~&~&~&~&0&0&1&1 \\ ~&~&~&~&~&~&~&~&~&0&0&-&- \\ ~&~&~&~&~&~&~&~&~&1&-&0&0&1&1 \\ ~&~&~&~&~&~&~&~&~&1&-&0&0&-&- \\~&~&~&~&~&~&~&~&~&~&~&1&-&x&-x \\ ~&~&~&~&~&~&~&~&~&~&~&1&-&-x&x \end{array}\right),$$

where $x$ is any unimodular number. To give a better understanding of the $E_{2m}$, here are the first three examples:

$$E_4(x)=\left(
\begin{array}{rrrr}
1 & 1 & 1 & 1 \\
1 & 1 & - & - \\
1 & - & x & -x \\
1 & - & -x & x
\end{array}
\right),
E_6(x)=\left(
\begin{array}{rrrrrr}
1 & 1 & 1 & 1 & 0 & 0\\
1 & 1 & - & - & 0 & 0\\
1 & - & 0 & 0 & 1 & 1\\
1 & - & 0 & 0 & - & -\\
0 & 0 & 1 & - & x & -x\\
0 & 0 & 1 & - & -x & x
\end{array}
\right),
$$
$$
E_8(x)=\left(
\begin{array}{rrrrrrrr}
1 & 1 & 1 & 1 & 0 & 0 & 0 & 0\\
1 & 1 & - & - & 0 & 0 & 0 & 0\\
1 & - & 0 & 0 & 1 & 1 & 0 & 0\\
1 & - & 0 & 0 & - & - & 0 & 0\\
0 & 0 & 1 & - & 0 & 0 & 1 & 1\\
0 & 0 & 1 & - & 0 & 0 & - & -\\
0 & 0 & 0 & 0 & 1 & - & x & -x\\
0 & 0 & 0 & 0 & 1 & - & -x & x
\end{array}
\right).
$$

To avoid extended case analysis, we provide only a subsection of the proof detailing the existence of a unique $UW(5,4)$ up to equivalence. The rest of the matrices follow using a similar manner of case analysis seen here and in section \ref{n3}.

\begin{lemma}
 \label{lem:cw-5-4}
 Let $\alpha = e^{\frac{2\pi i}{3}}$. Then every $UW(5,4)$ is equivalent to the following matrix.

$$
\left(
\begin{array}{ccccc}
 1 & 1       &1       &1        &0 \\
 1 & \alpha       &\overline{\alpha} &0        &1 \\
 1 & \overline{\alpha} &0       &\alpha       &\overline{\alpha} \\
 1 & 0       &\alpha       &\overline{\alpha}  &\alpha \\
 0 & 1       &\overline{\alpha} &\alpha        &\alpha
\end{array}
\right)
$$

\begin{proof}
By appropriate row and column permutations, we may assume that the zeros fall on the back diagonal of the matrix. We have the form:
   $$
\left(
\begin{array}{ccccc}
 1 & 1 &1 &1 &0 \\
 1 & a &b &0 &1 \\
 1 & c &0 &d &f \\
 1 & 0 &g &h &j \\
 0 & 1 &k &l &m
\end{array}
\right)
$$

Since row 1 must be orthogonal to each of the other rows, 3-orthogonality gives us:

   $$
\left(
\begin{array}{ccccc}
 1 & 1 &1 &1 &0 \\
 1 & a &\overline{a} &0 &1 \\
 1 & c &0 &\overline{c} &f \\
 1 & 0 &g &\overline{g} &j \\
 0 & 1 &k &\overline{k} &m
\end{array}
\right)
$$

where $a,c,f,j \in \BR$. Since column 1 must also be orthogonal to each of the other columns, 3-orthogonality gives us:

   $$
\left(
\begin{array}{ccccc}
 1 & 1 &1 &1 &0 \\
 1 & a       &\overline{a} &0 &1 \\
 1 & \overline{a} &0       &a &f \\
 1 & 0       &a       &\overline{a} &\overline{f} \\
 0 & 1       &k       &\overline{k} &m
\end{array}
\right)
$$

Since rows 2 and 3 must be orthogonal (and similarily, columns 2 and 3), we have:

$$
\left(
\begin{array}{ccccc}
 1 & 1 &1 &1 &0 \\
 1 & a       &\overline{a} &0 &1 \\
 1 & \overline{a} &0       &a &\overline{a} \\
 1 & 0       &a       &\overline{a} &a \\
 0 & 1       &\overline{a} &a &m
\end{array}
\right)
$$

Finally, since rows 4 and 5 must be orthogonal, we have $m=a$.

Note that we still have $a \in \{e^{\frac{2\pi i}{3}},e^{-\frac{2\pi i}{3}}\}$. But by interchanging the third and fourth rows and the second and third columns, we find that the choices are equivalent. For convenience, we assume $a=e^{\frac{2\pi i}{3}}$. Thus, we have the desired matrix.

\end{proof}

\end{lemma}

Altogether, we have that a $UW(n,4)$ exists for any $n \in \mathbb{N}, n \geq 4$. Moreover, the number of inequivalent $UW(n,4)$ is bounded by the number of decomposition of $n$ into sums of non-negative multiples of $5,6,7,8,$ and $2m$ (See Section \ref{sec:app} for the different combinations available for all $n \leq 14$).

If we are concerned with real matrices, then we may use only blocks of $W_7,W_8$ and $E_{2m}(1)$ (Note that $E_{2m}(1) \cong E_{2m}(-1)$ by swapping the last and second last columns). This implies that a $W(n,4)$ exists for any $n\neq5,9$. Moreover, the number of inequivalent $W(n,4)$ is bounded above by the number of decomposition of $n$ into sums of non-negative multiples of $7,8,$ and $2m$. 

To show  that $W_5, W_6, W_7$, $W_8 $ and $E_{2m}(x)$ are the only block formations that arise for $UW(n,4)$ takes a great deal of space. This is done by starting with the standard row of 4 ones, and then appending all possible rows in a depth first search manner. We omit the lengthy details. 

\subsection{Existence of $UW(5,5)$}

Haagerup\cite{H5} found that the only unit Hadamard matrix of order five is the Fourier matrix $F_5$ given here:
$$
\left(
\begin{array}{ccccc}
 1 & 1 &1 &1 &1 \\
 1 & \omega &\omega^2 &\omega^3 &\omega^4 \\
 1 & \omega^2 &\omega &\omega^4 &\omega^3 \\
 1 & \omega^3 &\omega^4 &\omega &\omega^2 \\
 1 & \omega^4 &\omega^3 &\omega^2 &\omega
\end{array}
\right)
$$ 
where $\omega$ is a primitive fifth root of unity.
\subsection{Existence of $UW(6,5)$}
The full analysis of $UW(6,5)$ is not yet complete. Thus far however, every matrix that we have found contains only the fourth root of unity or had only one free unimodular variable. For example,
$$\left(\begin{array}{rrrrrr}
             1 &  1       & 1 &  1 &         1 & 0 \\
             1 &  -       & x & -x &         0 & 1 \\
             1 & -\overline{x} & - &  0 &   \overline{x} & - \\
             1 & \overline{x}  & 0 &  - &  -\overline{x} & - \\
             1 &        0 & -x & x &         - & 1 \\
             0 &        1 &  - & - &         1 & 1
            \end{array}\right)$$
is a $UW(6,5)$ with one free variable $x \in \mathbb{T}$.

\subsection{Nonexistence of $UW(7,5)$}
\begin{lemma}
Any $UW(7,5)$ must include the following rows (after appropriate column permutations):
$$
\left(
\begin{array}{ccccccc}
 1&1&1&1&1&0&0 \\
 1&a&b&0&0&1&1 \\
 1&0&0&c&d&f&g \\
 0&0&1&h&k&m&n \\
 \end{array}
\right)
$$
\begin{proof}
To prove this condition, we show that three rows must exist with disjoint zeros (Two rows have disjoint zeros if for every column, there is at most one zero between the two rows).

Let $W$ be a $UW(7,5)$. We can begin by assuming the standard starting row of five 1's and two zeros. Permute the rows such that the second row is not disjoint from row 1. Two cases may occur from this: there is an overlap of either one or two zeros between the first and second rows. If there is an overlap of two zeros, then the third row must be disjoint from both the first and second rows. If there is single overlap, then permute the rows so that the third row has one overlap with the first. Then the fourth row must be disjoint from the first row. Thus, in either case, there are at least two disjoint rows.

From here, we can easily show that there must be three rows which are mutually disjoint. To do this, we assume that the first two rows are disjoint (say their zeros are in columns $1-4$). We may only put one more zero in each of those $4$ columns, but we have $5$ rows left, so at least one row must have no zeros in columns $1-4$. So this row, along with the first two, are mutually disjoint.

\end{proof}

\end{lemma}

\begin{theorem} \label{thm:vec_7_5}
 There is no $UW(7,5)$.

\begin{proof}
Any $UW(7,5)$ must contain the above vectors, which we will show cannot be mutually orthogonal.

Taking the pairwise standard complex inner product of the vectors, we obtain the following system of equations:

$$
\begin{cases}
 1+a+b &= ~~0 \\
 1+c+d &= ~~0 \\
 1+h+k &= ~~0 \\
 1+f+g &= ~~0 \\
 \overline{b}+m+n &= ~~0 \\
 h\overline{c}+k\overline{d}+m\overline{f}+n\overline{g} &= ~~0
\end{cases}
$$

The first 4 equations imply $a,b,c,d,f,g,h,k\in\{e^{\pm i\frac{2\pi}{3}}\}$ where $a,c,h,f$ are the conjugates of $b,d,k,g$ repectively. We will now re-write the vectors above.

$$
\left(
\begin{array}{ccccccc}
 1& 1& 1& 1& 1& 0& 0 \\
 1& a& \overline{a}& 0& 0& 1& 1 \\
 1& 0& 0& c& \overline{c}& f& \overline{f} \\
 0& 0& 1& h& \overline{h}& m& n
\end{array}
\right)
$$

Now let's consider the inner product of the second and fourth vectors: $a+m+n=0$. Since $a\in\{e^{\pm i\frac{2\pi}{3}}\}$, we have that $m,n \in \{1,\overline{a}\}$ where $m \neq n$. The inner product of rows 3 and 4 is now the sum of 4 third roots of unity, which cannot be zero. Thus, no $UW(7,5)$ can exist.

\end{proof}
\end{theorem}

\section{Appendices}\label{sec:app}

\begin{remark}\label{rem:weight-5-fix}
 In the course of our work, we have discovered that one matrix of order 12 and one matrix of order 14 were missing from the classification of weighing matrices of weight 5 in \cite{seberry}. Later on, we learned that Harada and Munemasa have also made note of this. We direct the reader to \cite[Section 4]{weigh-5-fix} for the exact details and the matrices that were missed.
\end{remark}

Given here is a list of unit matrices of weight 4. Recall that all unit weighing matrices of weight 4 are equivalent to a weighing matrix that is made up of $W_5$, $W_6$, $W_7$, $W_8$ and $E_{2m}(x)$. We now give examples of $UW(n,4)$ with $n$ small. To save space, we will denote the above matrices by $5^*,6^*,7^*,8^*$ and $2m$, respectively (any number without a ``$~^*$'' are of the form $2m$). Note that we make no claim about equivalence  of the matrices, only that this list is an upper bound on the number of inequivalent matrices.

\renewcommand{\tabcolsep}{.15cm}
\begin{tabular}[t]{ccc}
\multicolumn{3}{c}{Composition of Unit Weighing Matrices of type $UW(n,4)$}\\
\renewcommand{\tabcolsep}{.2cm}
\begin{tabular}{|l|l|}
 \hline $UW(4,4)$ &
 \begin{tabular}{ll}
 1. & 4  \\
 \end{tabular}

 \\ \hline $UW(5,4)$ &
 \begin{tabular}{ll}
 1. & 5*  \\
 \end{tabular}

 \\ \hline $UW(6,4)$ &
 \begin{tabular}{ll}
 1. & 6*  \\
 2. & 6  \\
 \end{tabular}

 \\ \hline $UW(7,4)$ &
 \begin{tabular}{ll}
 1. & 7*  \\
 \end{tabular}

 \\ \hline $UW(8,4)$ &
 \begin{tabular}{ll}
 1. & 8*  \\
 2. & 4 4  \\
 3. & 8  \\
 \end{tabular}
 
 \\ \hline $UW(9,4)$ &
 \begin{tabular}{ll}
 1. & 5* 4  \\
 \end{tabular}
 
 \\ \hline $UW(10,4)$ &
 \begin{tabular}{ll}
 1. & 5* 5*  \\
 2. & 6* 4  \\
 3. & 4 6  \\
 4. & 10  \\
 \end{tabular}
 
 \\ \hline $UW(11,4)$ &
 \begin{tabular}{ll}
 1. & 5* 6*  \\
 2. & 5* 6  \\
 3. & 7* 4  \\
 \end{tabular} \\
 \hline
 
 \end{tabular}
 %This is the separator for the LARGE table
& 
 \begin{tabular}{|l|l|}
 
 \hline $UW(12,4)$ &
 \begin{tabular}{ll}
 1. & 5* 7*  \\
 2. & 6* 6  \\
 3. & 6* 6*  \\
 4. & 8* 4  \\
 5. & 4 8  \\
 6. & 4 4 4  \\
 7. & 6 6  \\
 8. & 12  \\
 \end{tabular}
 
 \\ \hline $UW(13,4)$ &
 \begin{tabular}{ll}
 1. & 5* 8*  \\
 2. & 5* 4 4  \\
 3. & 5* 8  \\
 4. & 6* 7*  \\
5. & 7* 6  \\
\end{tabular}\\
\hline

\end{tabular}
 %This is the separator for the LARGE table
&
\begin{tabular}{|l|l|}
\hline $UW(14,4)$ &
\begin{tabular}{ll}
1. & 5* 5* 4  \\
2. & 6* 8*  \\
3. & 6* 4 4  \\
4. & 6* 8  \\
5. & 7* 7*  \\
6. & 8* 6  \\
7. & 4 10  \\
8. & 4 4 6  \\
9. & 6 8  \\
10. & 14  \\
\end{tabular}\\
\hline

\end{tabular}
\end{tabular}

The following table gives the number of decompositions of $n$ without showing the decompositions.

\begin{tabular}{|l|l||l|l||l|l||l|l|}
\hline
\multicolumn{8}{|c|}{Number of Decompositions}\\
\hline
$n$ & \# & $n$ & \# & $n$ & \# & $n$ & \# \\
\hline 
1 & 0 & 26 & 91 & 51 & 2401 & 76 & 49960 \\ 
2 & 0 & 27 & 73 & 52 & 3445 & 77 & 46836 \\ 
3 & 0 & 28 & 128 & 53 & 3089 & 78 & 61251 \\ 
4 & 1 & 29 & 103 & 54 & 4379 & 79 & 57587 \\ 
5 & 1 & 30 & 173 & 55 & 3952 & 80 & 74976 \\ 
6 & 2 & 31 & 142 & 56 & 5563 & 81 & 70630 \\ 
7 & 1 & 32 & 236 & 57 & 5034 & 82 & 91488 \\ 
8 & 3 & 33 & 194 & 58 & 7015 & 83 & 86422 \\ 
9 & 1 & 34 & 313 & 59 & 6391 & 84 & 111485 \\ 
10 & 4 & 35 & 265 & 60 & 8852 & 85 & 105496 \\ 
11 & 3 & 36 & 424 & 61 & 8082 & 86 & 135445 \\ 
12 & 8 & 37 & 357 & 62 & 11087 & 87 & 128477 \\ 
13 & 5 & 38 & 555 & 63 & 10177 & 88 & 164323 \\ 
14 & 10 & 39 & 476 & 64 & 13884 & 89 & 156137 \\ 
15 & 7 & 40 & 737 & 65 & 12778 & 90 & 198849 \\ 
16 & 16 & 41 & 634 & 66 & 17296 & 91 & 189343 \\ 
17 & 11 & 42 & 961 & 67 & 15987 & 92 & 240258 \\ 
18 & 23 & 43 & 837 & 68 & 21517 & 93 & 229138 \\ 
19 & 17 & 44 & 1256 & 69 & 19937 & 94 & 289613 \\ 
20 & 34 & 45 & 1098 & 70 & 26647 & 95 & 276750 \\ 
21 & 25 & 46 & 1621 & 71 & 24789 & 96 & 348615 \\ 
22 & 46 & 47 & 1433 & 72 & 32967 & 97 & 333611 \\ 
23 & 36 & 48 & 2102 & 73 & 30731 & 98 & 418702 \\ 
24 & 68 & 49 & 1860 & 74 & 40607 & 99 & 401394 \\ 
25 & 52 & 50 & 2687 & 75 & 37987 & 100 & 502179 \\
\hline
\end{tabular}

\section*{Acknowledgments}
The authors thank the referees for the helpful comments which
considerably improved the presentation of the paper.

\end{document}